\documentclass[11pt]{amsart}
\usepackage{amsmath,amssymb,amsthm,mathtools}
\usepackage[margin=1.05in]{geometry}

\usepackage{color}
\usepackage[pagebackref]{hyperref}
\hypersetup{
    colorlinks=true,
    citecolor=green, 
    linkcolor=blue, 
    urlcolor=blue   
}

\renewcommand*{\backref}[1]{}
\renewcommand*{\backrefalt}[4]{%
  \ifcase #1\relax
  \or\nobreakspace#2%
  \else\nobreakspace#2%
  \fi}

\usepackage{cleveref}

\usepackage{enumitem}
\usepackage{color}
\numberwithin{equation}{section}
\newtheorem{theorem}{Theorem}[section]
\newtheorem{mainthm}{Theorem}

\newtheorem{question}[theorem]{Question}
\newtheorem{proposition}[theorem]{Proposition}
\newtheorem{lemma}[theorem]{Lemma}

\theoremstyle{definition}
\newtheorem{definition}[theorem]{Definition}

\theoremstyle{remark}
\newtheorem{remark}[theorem]{Remark}

\newcommand{\m}{\mathfrak m}

\newcommand{\fa}{\mathfrak a}
\newcommand{\soc}{\operatorname{soc}}
\newcommand{\Hom}{\operatorname{Hom}}
\newcommand{\Ext}{\operatorname{Ext}}
\newcommand{\im}{\operatorname{im}}
\newcommand{\ann}{\operatorname{ann}}

\newcommand{\cT}{\mathcal T}
\newcommand{\cF}{\mathcal F}
\newcommand{\cG}{\mathcal G}
\newcommand{\cQ}{\mathcal Q}

\newcommand{\fB}{\mathfrak B}
\newcommand{\fu}{\mathfrak u}
\newcommand{\fv}{\mathfrak v}
\newcommand{\ip}{\iota}

\renewcommand{\ker}{\operatorname{ker}}

\newcommand{\id}{\operatorname{id}}

\newcommand{\cA}{\mathcal{A}}

\newcommand{\Z}{\mathbb{Z}}

\newcommand{\w}{\omega}

\renewcommand{\bar}{\overline}

\title[The residue field as an infinite syzygy]{The residue field as an infinite syzygy}
\date{\today}

\begin{document}

\author[J. Liu]{Jian Liu}
\email[Jian Liu]{jianliu@ccnu.edu.cn}
\urladdr{https://jianliu.org/}
\address{School of Mathematics and Statistics \\ and Hubei Key Laboratory of Mathematical
Sciences \\ Central China Normal University, Wuhan 430079, P. R. China.
}

\author[J. Lyle]{Justin Lyle}
\email[Justin Lyle]{jll0107@auburn.edu}
\urladdr{https://jlyle42.github.io/justinlyle/}
\address{Department of Mathematics and Statistics \\ 305 W Samford Avenue \\
Auburn University, AL 36849}



\begin{abstract}
Let $(R,\mathfrak{m},k)$ be a commutative Noetherian ring which is either local or positively graded over the field $k$. We provide a method based on liftings for constructing free coresolutions of $k$ over $R$. We apply this method to twisted trivial extensions of the node $k[u,v]/(uv)$ to build a family of positively graded $k$-algebras of dimension $1$ over which $k$ is an infinite graded syzygy, resolving an open question of Iyengar in the negative. We also prove a formula for alternating sums of lengths of Ext modules when $R$ is positively graded, and use it to show that a graded free coresolution of $k$ cannot dualize to a resolution unless $R$ is Artinian Gorenstein.

\end{abstract}
\maketitle

\section{Introduction}

Let $(R,\m,k)$ be either a commutative Noetherian local ring or a positively graded Noetherian $k$-algebra with homogeneous maximal ideal $\m$. In the graded setting, all ideals are understood to be homogeneous.

A finitely generated $R$-module $M$ is said to be an \emph{infinite syzygy} if there exists a free \emph{coresolution}
\[0 \to G^0 \to G^1 \to \cdots\]
of $M$. That is, $G^i$ is a finitely generated free $R$-module for all $i$, $H^0(G) \cong M$, and $H^i(G)=0$ for $i \ne 0$. When $R$ is graded, we say a finitely generated graded $R$-module $M$ is an \emph{infinite graded syzygy} if a coresolution can be chosen with graded free modules and degree zero differentials.

When $R$ is Artinian Gorenstein, a free coresolution is nothing but an injective resolution, and so in particular every finitely generated $R$-module is an infinite syzygy. The following question, generally attributed to Iyengar (see \cite{MRS} cf. \cite[Question 3.8]{Mohsen}), posits a strong converse:
\begin{question}[Iyengar]\label{quest:iyengar}
Suppose $k$ is an infinite syzygy over $R$. Must $R$ be an Artinian Gorenstein ring?
\end{question}

In this work we give a general mechanism for building positive-dimensional examples over which $k$ is an infinite syzygy, and use it to resolve Question \ref{quest:iyengar} in the negative:

\begin{mainthm}\label{thm:maincriterion}
Suppose $(R,\m,k)$ is a local ring (resp.~positively graded $k$-algebra) of positive dimension. Suppose $M$ is a nonzero ideal of $R$ with finite length as an $R$-module, set $B:=R/\ann_R(M)$, and let $K$ be the minimal (graded) free resolution of $k$ over $B$. Suppose that there is a (graded) lift of $K$ to $R$, that is, a (graded) complex $\cT$ of free $R$-modules with $B\otimes_R\cT\cong K$, such that the maps
\[
\rho_i\colon H^i(\Hom_R(\cT,R))\longrightarrow\Ext_B^i(k,B)
\]
induced by reduction modulo $\ann_R(M)$ satisfy $\operatorname{length}_R \ker\rho_0=1$ and are injective for $i\ge1$. Then $k$ is an infinite (graded) syzygy over $R$.
\end{mainthm}

We present a family of graded examples for which we establish the lifting condition of Theorem \ref{thm:maincriterion}, answering Iyengar's question negatively in embedding dimension four. This is the smallest possible embedding dimension, as Miller--Rahmati--Striuli \cite[Corollary~2.13]{MRS} show that if $R$ has embedding dimension at most three and $k$ is even a third syzygy, then $R$ is Gorenstein. 

The family is given as follows: Let $k$ be a field and let $q=(q_{ij})$ be a symmetric invertible $2\times2$ matrix over $k$. Let
\[
R(q):=k[u,v,x,y]\big/\bigl(ux,\;uy,\;vx,\;vy,\;x^2-q_{11}uv,\;xy-q_{12}uv,\;y^2-q_{22}uv\bigr).
\]

\begin{mainthm}\label{thm:node}
$R(q)$ is a standard graded Koszul $k$-algebra of dimension one and depth zero, and there is a minimal graded free linear coresolution
\[
0\to k\langle2\rangle\to R(q)\to R(q)\langle-1\rangle^{\beta_1}\to R(q)\langle-2\rangle^{\beta_2}\to\cdots,
\]
where $\displaystyle \sum_{n\ge0}\beta_nz^n=\frac{1+z}{1-z-3z^2+z^3}$.
\end{mainthm}

Each $R(q)$ is a \emph{twisted trivial extension} $S\ltimes_b\omega_B\langle2\rangle$, with $S=k[U,V]/(UV)$ and $B=S/(U,V)^2$; see \Cref{subsec:Rq}. Here a symmetric form $b$ modifies the usual square-zero multiplication on $\omega_B\langle2\rangle$ by allowing products in its one-dimensional socle. The case $b=0$ is the usual trivial extension, over which $k$ is not even a third syzygy in this example (see \Cref{prop:nodetwist}).

When $q$ is the identity matrix this gives the particularly concrete ring
\[R(\id)=k[u,v,x,y]/(ux,\,uy,\,vx,\,vy,\,xy,\,x^2-uv,\,y^2-uv).
\]

For this class of rings, we note that $M=H^0_\m(R(q))=(x,y)$ while $B=k[x,y]/(x,y)^2$. A Clifford module on the free algebra $\Ext_B(k,k)=k\langle\xi,\eta\rangle$ supplies a lift satisfying the hypotheses of Theorem~\ref{thm:maincriterion}; see \Cref{clifford}.

While Iyengar's question was originally about local rings, the local variation follows at once from the graded one, by either localization or completion. The graded setting offers meaningful advantages. Perhaps the simplest method via which one could hope to realize $k$ as an infinite syzygy would be to find a ring $R$ and a finitely generated $R$-module $N$ with $\Hom_R(N,R) \cong k$ and $\Ext^{i>0}_R(N,R)=0$. In this situation, if $F$ is a minimal free resolution of $N$, then $\Hom_R(F,R)$ will be a minimal free coresolution of $k$. In the graded setting, we prove a general result (\Cref{thm:extlength}) on alternating sums of lengths of $\Ext$ modules and use it to show existence of such a module already forces $R$ to be Artinian Gorenstein, though the local variation remains open:

\begin{mainthm}\label{thm:introC}
Suppose $(R,\m,k)$ is a positively graded $k$-algebra. If there is a finitely generated graded $R$-module $N$ and an integer $\ell$ such that $\Hom_R(N,R) \cong k\langle \ell \rangle$ and $\Ext^{i>0}_R(N,R)=0$, then $R$ is an Artinian Gorenstein ring. 
\end{mainthm}

We now outline the structure of this work. Sections~\ref{sec:setting} and~\ref{sec:criterion} set up the framework and prove Theorem~\ref{thm:maincriterion} in a more general form, together with a partial converse. Section~\ref{sec:node} gives the main application, proving Theorem~\ref{thm:node}. Section~\ref{sec:graded} proves \Cref{thm:extlength} which gives a formula for alternating sums of lengths of $\Ext$ modules, and uses it to prove Theorem~\ref{thm:introC}.

\section{Preliminaries}\label{sec:setting}

Throughout, all rings are commutative and Noetherian. We let $(R,\m,k)$ denote either a local ring or a positively graded $k$-algebra with maximal homogeneous ideal $\m$. Unless otherwise indicated, all modules are assumed to be finitely generated and graded in the graded setting. In this setting, with $N$ a graded $R$-module, we let $N\langle j \rangle$ denote the $j$th graded shift of $N$, so $N\langle j \rangle_i=N_{i-j}$, and we let $(-)^{\vee}:=\Hom_k(-,k)$ denote the $k$-dual functor.

\subsection{Standing hypotheses and notation}\label{subsec:framework}

Let $M\subseteq\m$ be a nonzero ideal of $R$ with finite length as an $R$-module and suppose $\soc M=km_0$ for some $m_0$. Set $\fa=\ann_R(M)$, $B=R/\fa$, $S=R/M$, and let $\pi\colon R\to B$ be the natural projection. Since the ideal $M$ is nilpotent, we note that $\dim R=\dim S$ and $\operatorname{depth}R=0$. Note that $m_0 \in \fa$. Since $M$ has finite length, the ring $B$ is Artinian, with maximal ideal $\m_B=\m/\fa$. Since $M$ is a faithful $B$-module of type $1$, it follows from \cite[Proposition 3.2.12]{BH} that  $M\cong E_B(k)$ over $B$, where $E_B(k)$ is the injective hull of $k$. Write $\bar{(-)}:=B \otimes_R -$, and let $K$ be the minimal free resolution of $k$ over $B$, with differential $\partial$.

The hypotheses of Theorem~\ref{thm:maincriterion} fit this framework automatically. Indeed, positive dimension implies $M\subseteq\m$, and minimality of $K$ implies that the lift $\cT$ is minimal. Writing $I$ for the image of $\cT_1\to\cT_0=R$, we have $I\subseteq\m$ and
\[
0\ne\soc M\subseteq\soc R\subseteq(0:_RI)\cap\ann_R(M)=\ker\rho_0.
\]
The length hypothesis therefore gives $\soc M=\soc R=\ker\rho_0=km_0$ for some $m_0$.

In the graded setting, if we take $\w_B:=B^{\vee}$, which is isomorphic to $E_B(k)$ up to graded shift, then it follows from above that $M \cong \w_B\langle \ell \rangle$. With $\soc(M)=km_0$, it follows that $m_0 \in M_{\ell}$ and is dual to $1$ under the isomorphism $M^{\vee} \cong B\langle -\ell\rangle$. We let $\sigma:M \to k$ be the degree $-\ell$ map given as the composition $M \cong \w_B\langle \ell \rangle \to k$ where the second map is given by evaluation at $1$.

\subsection{Duality and reduction}\label{subsec:dual}
Matlis duality over $B$ makes $\mathsf D(-):=\Hom_B(-,M)$ an exact, length-preserving duality on finitely generated $B$-modules.

In the graded setting, we have $\mathsf D(X)\cong X^\vee\langle\ell\rangle$ by $f\mapsto\sigma\circ f$.

For a complex $X$ of finite free $B$-modules, evaluation induces
\begin{equation}\label{eq:homologypairing}
H_n(M\otimes_BX)\cong
\Hom_B\bigl(H^n(\Hom_B(X,B)),M\bigr),
\end{equation}
with pairing $([z],[\psi])\mapsto(1_M\otimes\psi)(z)$. When $X=K$, the $\Ext_B^n(k,B)$ are annihilated by $\m_B$, so this pairing takes values in $\soc M=km_0$ and is perfect over $k$. For a free $R$-module $F$, note that multiplication identifies $M\otimes_B\bar F$ with $MF\subseteq F$, and $\Hom_R(F,M)=\Hom_B(\bar F,M)$.

Set $(-)^*:=\Hom_R(-,R)$.

If $\cF$ is a complex of free $R$-modules concentrated in nonnegative homological degrees, then applying $\Hom_R(\cF,-)$ to the short exact sequence $0\to M\to R\to S\to0$ gives a short exact sequence of complexes
\begin{equation}\label{eq:ses}
0\longrightarrow\Hom_B(\bar\cF,M)\longrightarrow\cF^*\longrightarrow\Hom_R(\cF,S)\longrightarrow0.
\end{equation}
The first term has cohomology $\Hom_B(H_n(\bar\cF),M)$. The connecting map
\[
\delta_n\colon H^n(\Hom_R(\cF,S))\longrightarrow\Hom_B(H_{n+1}(\bar\cF),M)
\]
is computed by lifting a cocycle in $\Hom_R(\cF,S)^n$ to $\varphi\in(\cF^*)^n$. Then $\varphi d_{n+1}$ takes values in $M$ and represents the connecting class in $\Hom_B(\bar\cF,M)^{n+1}$.

In particular, if $\bar\cT\cong K$, then the map induced by $R\to S$ gives
\begin{equation}\label{eq:liftcomparison}
H^n(\Hom_R(\cT,S))\cong
\begin{cases}
H^0(\cT^*)/km_0,&n=0,\\
H^n(\cT^*),&n\ge1,
\end{cases}
\end{equation}
since $H^n(\Hom_B(K,M))$ is $km_0$ for $n=0$ and zero otherwise.

\section{A criterion for a free coresolution}\label{sec:criterion}

In this section we prove Theorem \ref{thm:maincriterion}. We keep the notation of \Cref{sec:setting}.

\begin{definition}\label{def:lift}
A \emph{lift of $K$ along $\pi$} is a complex $\cT$ of free $R$-modules with $\bar\cT\cong K$. We note $\cT_0=R$, and we write $d$ for its differential. We write $\bar x$ for the image of $x\in\cT$. Since $\pi^{-1}(\m_B)=\m$, a lift is necessarily minimal. We say the lift $\cT$ is \emph{dually small} if the induced maps
\[
\rho_m\colon H^m(\cT^*)\longrightarrow\Ext^m_B(k,B),
\qquad[\varphi]\longmapsto[\bar\varphi],
\]
are injective for $m\ge1$ and have $\ker\rho_0=k m_0$.
\end{definition}

Note the element $m_0$ always lies in $\ker\rho_0$, since $\m m_0=0$ and $\pi(m_0)=0$. \eqref{eq:liftcomparison} identifies $H^m(\Hom_R(\cT,S))$ with $H^m(\cT^*)$ for $m\ge1$ and with $H^0(\cT^*)/k m_0$ for $m=0$. Thus the dually small condition indicates the collective kernel of the $\rho_m$ is as small as possible, and in particular embeds these spaces in $\Ext_B^m(k,B)$. 

The following is the main theorem of this section:
\begin{theorem}\label{thm:U}
Suppose that $K$ has a lift $\cT$ along $\pi$ that is dually small. Then $k$ is an infinite syzygy over $R$, and is an infinite graded syzygy in the graded setting.
\end{theorem}

For the proof, fix a dually small lift $\cT$. We construct a complex $\cF$ whose $R$-dual is the required coresolution. Modulo $M$, the construction is a direct sum of shifted copies of $\cT/M\cT$. Modulo $\fa$, an explicit change of basis identifies it with a direct sum of shifted copies of $K$. Correction terms make the connecting maps of \eqref{eq:ses} isomorphisms. Proving \Cref{thm:U} will in particular require several preparatory results, established over the next few subsections.

\subsection{Products and detecting cycles}\label{subsec:products}\label{subsec:letters}
By a theorem of Gulliksen \cite{Gulliksen} and Schoeller \cite{Schoeller} (see also \cite[Theorem~6.3.5]{Avramov} cf. \cite{Tate}) $K$ is the acyclic closure of $k$ over $B$, and so in particular possesses a strictly graded-commutative dg $B$-algebra structure which is compatible with the internal degree in the graded setting.

The complex $M\otimes_BK$, with differential $\tau:=1_M\otimes\partial$, is a DG $K$-module in the usual way, and $\pi|_M$ induces a morphism of DG $K$-modules
\[
\pi_K\colon M\otimes_BK\longrightarrow K,\qquad
\mu\otimes y\longmapsto\pi(\mu)y.
\]

Multiplication gives an injective chain map
$\ip\colon M\otimes_BK\to\cT$, with $\mu\otimes\bar x\mapsto\mu x$, and with image $M\cT$. It satisfies $\overline{\ip(z)}=\pi_K(z)$ and, for $\varphi\in(\cT^*)^m$, we have
\begin{equation}\label{eq:phiiota}
\varphi(\ip(z))=(1_M\otimes\bar\varphi)(z)
\qquad(z\in M\otimes_BK_m).
\end{equation}
Indeed, $\varphi(\mu x)=\mu\varphi(x)=\mu\pi(\varphi(x))$.

For each $m\ge0$, choose cocycles $\varphi_v\in(\cT^*)^m$, indexed by a set $\fB_m$, lifting a $k$-basis of $\im\rho_m$. Under the dually small condition, their classes form a basis of $H^m(\cT^*)$ for $m>0$, and their images form a basis of $H^0(\cT^*)/km_0$ for $m=0$. Each $\fB_m$ is finite, and the classes $[\bar\varphi_v]$ are linearly independent. The duality \eqref{eq:homologypairing} gives cycles $z_v\in M\otimes_BK_m$ with
\begin{equation}\label{eq:dual}
(1_M\otimes\bar\varphi_u)(z_v)=\delta_{uv}m_0
\qquad(u,v\in\fB_m).
\end{equation}
The cycle $\pi_K(z_v)$ is a boundary in $K$: for $m>0$ this follows from exactness, and for $m=0$ from $\pi(M)\subseteq\m_B=\partial(K_1)$. For each $v$, choose $c_v\in K_{m+1}$ with $\partial c_v=\pi_K(z_v).$

Regard $\fB=\bigsqcup_m\fB_m$ as an alphabet. A letter $v\in\fB_m$ has length $|v|=m+1$, and we extend length additively to words, with $|\varnothing|=0$ for the empty word. There are finitely many words of length at most $n$ for each $n$. For a nonempty word $s=v_1\cdots v_r$, set
\[
t(s)=z_{v_1}c_{v_2}\cdots c_{v_r}\in M\otimes_BK_{|s|-1}.
\]

\begin{lemma}\label{lem:MC}
For every nonempty word $s$, we have
\[
\tau t(s)+\sum_{s=s's''}(-1)^{|s'|}
 t(s')\pi_K(t(s''))=0,
\]
where both factors in the sum are nonempty.
\end{lemma}
\begin{proof}
Write $s=v_1\cdots v_r$. Since $z_{v_1}$ is a cycle, the Leibniz rule gives
\[
\tau t(s)=-\sum_{i=2}^r(-1)^{|v_1\cdots v_{i-1}|}
 z_{v_1}c_{v_2}\cdots c_{v_{i-1}}\pi_K(z_{v_i})c_{v_{i+1}}\cdots c_{v_r}.
\]
For the cut $s' =v_1\cdots v_{i-1}$, $s''=v_i\cdots v_r$, the product in this sum is $t(s')\pi_K(t(s''))$.
\end{proof}

\subsection{The complex}\label{subsec:complex}
Let $V$ be the free $R$-module with basis $\fB$, placing $v$ in homological degree $|v|$. The tensor algebra
\[
W=T_R(V)=\bigoplus_{r\ge0}V^{\otimes_R r}
\]
has the words in $\fB$ as an $R$-basis, with the empty word corresponding to $1$. Set $\cF=W\otimes_R\cT$ as a tensor product of homologically graded $R$-modules, so that
\[
\cF_n \cong \bigoplus_{|w|\le n}\cT_{n-|w|}.
\]
Define
\begin{equation}\label{eq:d}
d(w\otimes x)=(-1)^{|w|}w\otimes dx
 +\sum_{w=w's}(-1)^{|w'|}w'\otimes\ip(t(s)\bar x),
\end{equation}
where $s$ is nonempty. The first term is the ordinary tensor-product differential when $W$ has zero differential, while the sum can be thought of as adding correction terms to this differential.

\begin{lemma}\label{lem:complex}
The maps \eqref{eq:d} make $\cF$ a minimal complex of finite free $R$-modules, with $\cF_0=R$.
\end{lemma}
\begin{proof}
The maps are $R$-linear and lower homological degree by one. Their entries lie in $\m$, since $\cT$ is minimal and the correction terms lie in $M\cT$.

Fix $w=w'u$ with $u$ nonempty. The component of $d^2(w\otimes x)$ on $w'$ is $\ip$ applied to
\[
(-1)^{|u|}t(u)\partial\bar x+\tau(t(u)\bar x)
 +\sum_{u=s's}(-1)^{|s'|}t(s')\pi_K(t(s))\bar x.
\]
Here we used that $\ip$ is a chain map and $\overline{\ip(z)}=\pi_K(z)$. The Leibniz rule cancels the two terms containing $\partial\bar x$, and Lemma~\ref{lem:MC} cancels the rest. The component on $w$ is $d^2x=0$.
\end{proof}

\subsection{Reduction and dual exactness}\label{subsec:connecting}
Reduction modulo $\fa$ gives
\[
\bar d(w\otimes y)=(-1)^{|w|}w\otimes\partial y
 +\sum_{w=w's}(-1)^{|w'|}w'\otimes\pi_K(t(s))y.
\]
Let $\cG=\bigoplus_w\Sigma^{|w|}K$, where $\Sigma$ denotes homological shift, and write $w\otimes y$ for $y$ in the $w$-summand. Its differential is $d'(w\otimes y)=(-1)^{|w|}w\otimes\partial y$. The following change of basis identifies $\cG$ with $\bar\cF$:

\begin{lemma}\label{lem:barF}
We have the following:
\begin{enumerate}
\item[$(1)$] The map
\[
\Theta(wv\otimes y)=wv\otimes y-w\otimes c_vy,
\qquad \Theta(\varnothing\otimes y)=\varnothing\otimes y,
\]
where $v$ is a letter, is a chain isomorphism $\Theta\colon\cG\to\bar\cF$.
\item[$(2)$] Set $\zeta_\varnothing=\varnothing\otimes1$ and
\[
\zeta_{wv}=wv\otimes1-w\otimes c_v.
\]
The classes $[\zeta_w]$ with $|w|=n$ form a $k$-basis of $H_n(\bar\cF)$.
\item[$(3)$] For $|w|=n$, let $\eta_w\colon\bar\cF_n\to M$ be supported on $w\otimes K_0$ and send $w\otimes1$ to $m_0$. The classes $[\eta_w]$ form a basis of $H^n(\Hom_B(\bar\cF,M))$. The coefficient $\lambda_w\in k$ of $[\eta_w]$ in the class of a cocycle $\xi$ is determined by $\xi(\zeta_w)=\lambda_wm_0$.
\end{enumerate}
\end{lemma}
\begin{proof}
(1) For $y\in K_j$, both terms defining $\Theta(wv\otimes y)$ have homological degree $|wv|+j$. In $\bar d\Theta(wv\otimes y)$, the term arising from the cut with suffix $v$ cancels the term containing $\partial c_v=\pi_K(z_v)$. The remaining correction terms cancel in pairs because $t(sv)=t(s)c_v$. The Leibniz rule leaves
\[
\bar d\Theta(wv\otimes y)
 =(-1)^{|wv|}\bigl(wv\otimes\partial y-w\otimes c_v\partial y\bigr)
 =\Theta d'(wv\otimes y).
\]
The assertion for the empty word is immediate. Under the natural identification of the underlying homologically graded modules, write $\Theta=1-C$. The map $C$ reduces the number of letters, so $\sum_{j\ge0}C^j$ is a well-defined inverse, finite on each element.

(2) The homology of $\cG$ has basis $[w\otimes1]$, and $\Theta(w\otimes1)=\zeta_w$, so the claim follows.

(3) The complex $\bar\cF$ is minimal and $\m_Bm_0=0$, so each $\eta_w$ is a cocycle. For words of length $n$, one has $\eta_{w'}(\zeta_w)=\delta_{ww'}m_0$, as the correction in $\zeta_w$ lies on a shorter word. By part (2) and Matlis duality, these cocycles give a basis, with the stated coefficient formula.
\end{proof}

For $\varphi\in(\cT^*)^j$, let $\varphi^{(w)}\in(\cF^*)^{|w|+j}$ be supported on the $w$-summand and equal there to $\varphi$.
\begin{lemma}\label{lem:connecting}
The images of $\varphi_v^{(w)}$, with $|w|+|v|=n+1$, represent a basis of $H^n(\Hom_R(\cF,S))$. On this basis, the connecting map of \eqref{eq:ses} is
\begin{equation}\label{eq:connectingletters}
\delta_n[\varphi_v^{(w)}]=(-1)^{|w|}[\eta_{wv}].
\end{equation}
\end{lemma}
\begin{proof}
Modulo $M$, the correction terms of \eqref{eq:d} vanish. Thus $\Hom_R(\cF,S)$ is the direct sum of the shifted complexes $\Hom_R(\cT,S)$, and the basis assertion follows from \eqref{eq:liftcomparison} and the choice of the $\varphi_v$.

The connecting class is represented by $\xi=\varphi_v^{(w)}d\in\Hom_B(\bar\cF,M)^{n+1}$. Since $\varphi_v$ is a cocycle, $\xi$ vanishes outside words $ws$ with $s$ nonempty. On that summand its value at $ws\otimes y$ is
\[
(-1)^{|w|}(1_M\otimes\bar\varphi_v)(t(s)y).
\]
Evaluate on $\zeta_{pq}=pq\otimes1-p\otimes c_q$, where $q$ is a letter and $|pq|=n+1$. If $p=ws$ with $s$ nonempty, the two contributions cancel as $t(sq)=t(s)c_q$. If $p=w$, only the first term contributes, and \eqref{eq:dual} gives
\[
\xi(\zeta_{wq})=(-1)^{|w|}\delta_{vq}m_0.
\]
All other cases give zero, and Lemma~\ref{lem:barF}(3) proves \eqref{eq:connectingletters}.
\end{proof}
We are now ready for the proof of Theorem~\ref{thm:U} which also covers Theorem \ref{thm:maincriterion} from the intro:
\begin{proof}[Proof of Theorem~\ref{thm:U}]
Every nonempty word has a unique final letter, so \eqref{eq:connectingletters} makes each $\delta_n$ an isomorphism. The long exact sequence of \eqref{eq:ses} gives $H^n(\cF^*)=0$ for $n\ge1$ and identifies $H^0(\cF^*)$ with the image of
$H^0(\Hom_B(\bar\cF,M))=k[\eta_\varnothing]$. This image is $k m_0\subseteq\cF_0^*=R$. Taking $G^n=\cF_n^*$ gives the desired coresolution, which is minimal because $\cF$ is minimal. 

In the graded setting, the coresolution is graded provided we choose the $\varphi_v$, $z_v$, and $c_v$ homogeneously. The perfect pairing gives $\deg z_v=\ell-\deg\varphi_v$, and we choose $\deg c_v=\deg z_v$. Give each letter weight $\|v\|=\deg z_v$ and extend weight additively to words. Then $\deg t(s)=\|s\|$, and the word construction, with these internal degrees, has
\[
\cF_n=\bigoplus_{|w|\le n}\cT_{n-|w|}\langle\|w\|\rangle 
\qquad \text{with} \qquad G^n=\cF_n^*.
\]
With these internal degrees, $\cG=\bigoplus_w(\Sigma^{|w|}K)\langle\|w\|\rangle$. Every term of \eqref{eq:d} has internal degree zero, as does the change of basis $\Theta$. The element $\varphi^{(w)}$ has internal degree $\deg\varphi-\|w\|$, and the classes $[\zeta_w]$ and $[\eta_w]$ have degrees $\|w\|$ and $\ell-\|w\|$, respectively. Thus all the preceding identifications respect the internal grading.
\end{proof}

\begin{remark}\label{rem:untwisted}
If $M^2=0$, then $\pi_K$ vanishes on $M\otimes_BK$, so we may take every $c_v=0$. Thus $t$ vanishes on words with at least two letters, $\Theta$ is the identity, and the sum in \eqref{eq:d} has at most one term. For nonzero $M^2$, the longer-word terms enforce $d^2=0$, while the explicit change of basis above makes the connecting maps just as simple.
\end{remark}

\subsection{A partial converse}\label{subsec:converse}
We maintain the hypotheses of \Cref{subsec:framework}. This subsection extracts some necessary conditions for $k$ to be a third syzygy that will be needed in the next section.

\begin{proposition}\label{prop:necU}
Suppose that $\operatorname{depth}S>0$. If $k$ is a third syzygy over $R$, then there is an ideal $I\subseteq\m$, homogeneous in the graded setting, such that
\[
I+\fa=\m\qquad\text{and}\qquad I\cap\fa=I\fa.
\]
If moreover $\m_B^2=0$, then $\m^2=\fa\m$.
\end{proposition}
\begin{proof}
An exact sequence witnessing that $k$ is a third syzygy may be chosen minimal, and \cite[Lemma~2.8]{MRS} implies that, after an overall shift in the graded setting, it has the form
\[
0\longrightarrow km_0\hookrightarrow G^0
\longrightarrow G^1\longrightarrow G^2,
\]
with $G^0=R$ and each $G^i$ free. Let $\cF$ be the $R$-dual of $G$. Then $\cF_0=R$, $H^0(\cF^*)=km_0$, and $H^1(\cF^*)=0$.

Set $I=d(\cF_1)\subseteq\m$ and let $IS$ be its image in $S$. Since $(0:_RI)=H^0(\cF^*)=km_0\subseteq M$, we have
\[
\Hom_B(B/\pi(I),M)=(0:_MI)=km_0.
\]
Duality gives $B/\pi(I)\cong k$, proving $I+\fa=\m$. The map $H^0(\Hom_B(\bar\cF,M))\to H^0(\cF^*)$ in \eqref{eq:ses} is therefore an isomorphism. Since $H^1(\cF^*)=0$, the relevant long exact sequence gives
\[
\delta_0\colon(0:_SIS)\xrightarrow{\ \cong\ }\Hom_B(H_1(\bar\cF),M).
\]
The displayed isomorphism shows that the ideal $(0:_SIS)$ has finite length, so it vanishes since $\operatorname{depth}S>0$. Duality therefore gives $H_1(\bar\cF)=0$. Thus $d(y)\in\fa$ implies $y\in\im d_2+\fa\cF_1$, whence $d(y)\in I\fa$, which proves $I\cap\fa=I\fa$.

Finally, if $\m_B^2=0$, then $\m^2\subseteq\fa$, so
\[
I\m\subseteq I\cap\fa=I\fa\subseteq\fa\m,
\qquad\text{and hence}\qquad
\m^2=(I+\fa)\m=\fa\m.
\]
\end{proof}

\section{Twisted trivial extensions of the node}\label{sec:node}

In this section we prove Theorem~\ref{thm:node}.

\subsection{\texorpdfstring{The rings $R(q)$}{The rings R(q)}}\label{subsec:Rq}
Fix a symmetric $2\times2$ matrix $q=(q_{ij})$ over $k$ with $\det q\ne0$. Write $x_1=x$, $x_2=y$, and set
\[
R(q)=k[u,v,x,y]\big/\bigl(ux,\;uy,\;vx,\;vy,\;x_ix_j-q_{ij}\,uv\ (1\le i\le j\le2)\bigr),
\]
and set $m_0=uv$.

The defining relations give the graded pieces
\begin{equation}\label{eq:Rqbasis}
\begin{aligned}
R(q)_0&=k,& R(q)_1&=ku\oplus kv\oplus kx\oplus ky,\\
R(q)_2&=ku^2\oplus kv^2\oplus km_0,&
R(q)_n&=ku^n\oplus kv^n\quad(n\ge3).
\end{aligned}
\end{equation}
Abbreviating $R:=R(q)$, we have
\[
\begin{gathered}
M=(x,y)=kx\oplus ky\oplus km_0,\qquad \ann_R(M)=(u,v),\\
B=R/(u,v)=k[x,y]/(x,y)^2,\qquad \text{and} \quad S=R/M=k[u,v]/(uv).
\end{gathered}
\]
The coefficient of $m_0$ in $r\mu$ defines a pairing $B\times M\to k$ with matrix $1\oplus q$ in the bases $(1,\bar x,\bar y)$ of $B$ and $(m_0,x,y)$ of $M$, so $M\cong\omega_B\langle2\rangle$. Thus the graded hypotheses of \Cref{subsec:framework} hold with $\ell=2$. Since $M$ is nilpotent and $\m m_0=0$, the ring $R$ has dimension one and depth zero.

These rings are twisted trivial extensions of the node. To see this, present $S=k[U,V]/(UV)$, put $B'=S/(U,V)^2$, and let $N=\omega_{B'}\langle2\rangle$. Choose basis elements $X,Y$ in degree one dual to $\bar U,\bar V$, and let $z$ in degree two be dual to $1$, so that
\[
UX=VY=z,\qquad UY=VX=0.
\]
A \emph{twist} is a symmetric form $b=(b_{ij})$ on $N_1=kX\oplus kY$, extended to $N$ by $b(z,N)=0$. The \emph{twisted trivial extension} $R_b=S\ltimes_bN$ is $S\oplus N$ as an $S$-module, with multiplication
\[
(s,\mu)(s',\mu')=\bigl(ss',\;s\mu'+s'\mu+b(\mu,\mu')z\bigr).
\]
Thus $X^2=b_{11}z$, $XY=b_{12}z$ and $Y^2=b_{22}z$; when $b=0$, this is the usual trivial extension. The ring $R_b$ is standard graded, with $N$ a finite-length ideal satisfying $N^2\subseteq kz$, $N^3=0$ and $R_b/N\cong S$. Its action on $N$ is given by the graded projection $\pi\colon R_b\to B'$ with $\pi(z)=0$ and
\[
\begin{aligned}
\pi(U)&=\bar U,&\qquad \pi(V)&=\bar V,\\
\pi(X)&=b_{11}\bar U+b_{12}\bar V,&
\pi(Y)&=b_{12}\bar U+b_{22}\bar V.
\end{aligned}
\]
In particular, $\ker\pi=\ann_{R_b}(N)$, and the hypotheses of \Cref{subsec:framework} hold with $M=N$ and $m_0=z$.

If $\det b\ne0$ and $b_{12}\ne0$, then, with $\hat b=b^{-1}$, the assignment
\begin{equation}\label{eq:nodepresentation}
\begin{aligned}
u&\longmapsto\frac{\det b}{b_{12}}\bigl(U-\hat b_{11}X-\hat b_{12}Y\bigr),& x&\longmapsto X,\\
v&\longmapsto V-\hat b_{12}X-\hat b_{22}Y,& y&\longmapsto Y
\end{aligned}
\end{equation}
induces an isomorphism $R(b)\to R_b$. Indeed, writing $u'=U-\hat b_{11}X-\hat b_{12}Y$ and $v'=V-\hat b_{12}X-\hat b_{22}Y$, we have $u'N=v'N=0$ and $u'v'=(b_{12}/\det b)z$, so the map respects the relations and is surjective. Both rings have the Hilbert function in \eqref{eq:Rqbasis}, so it is an isomorphism. It carries $(x,y)$ to $N$ and their annihilators to one another, identifying $B$ with $B'$ through the action on $N$.

A linear change of $(x,y)$ replaces $q$ by a congruent matrix. If $q_{12}=0$, then $q_{22}\ne0$, and the basis $x+y,y$ gives the off-diagonal entry $q_{22}$. Thus every $R(q)$ has a twisted presentation as above with $\det b\ne0$ and $b_{12}\ne0$.

The rings $R(q)$ are also Koszul. By \cite[Main Theorem]{AMS}, a quadratic standard graded $k$-algebra $A$ with $\dim_kA_2\le3$ can fail to be Koszul only if $\dim_kA_3=1$: the exceptional algebras have Hilbert series $(1+2t-2t^3)/(1-t)$ up to trivial fiber extensions, which change only degree one. Since $\dim_kR(q)_2=3$ and $\dim_kR(q)_3=2$ by \eqref{eq:Rqbasis}, the claim follows.

\subsection{A lift from a Clifford module}\label{clifford}
Let $\cA=k\langle\xi,\eta\rangle$ be the free associative algebra, graded by word length, and write $\xi_1=\xi$, $\xi_2=\eta$. For $a\in \cA$, let $\mathrm r_a\colon \cA\to \cA$ denote right multiplication by $a$, and set
\[
\xi_i'=\sum_jq_{ij}\,\xi_j\quad(i=1,2) \qquad \text{and} \qquad \chi=\sum_{i,j}q_{ij}\,\xi_i\xi_j=\xi_1'\xi+\xi_2'\eta\in \cA_2 .
\]
The elements $\xi_1',\xi_2'$ form a basis of $\cA_1$, and $\chi\ne0$ because the length-two words are linearly independent. For each $n$, let $V_n$ be the vector-space dual of $\cA_n$, placed in internal degree $n$.

The ring $B=k[x,y]/(x,y)^2$ is Koszul, with $\Ext_B(k,k)=B^!=\cA$ and $K_n\cong B\otimes_kV_n$ (see \cite{Lofwall,PP,Priddy}). Since $\m_B^2=0$, this is also the normalized bar resolution. Replacing $\bar x,\bar y$ by $x,y$ in its differential gives
\[
\begin{aligned}
\widetilde\partial(1\otimes\lambda)&=x\otimes\lambda\circ\mathrm r_\xi+y\otimes\lambda\circ\mathrm r_\eta,\\
\widetilde\partial^{\,2}(1\otimes\lambda)&=m_0\otimes\lambda\circ\mathrm r_\chi.
\end{aligned}
\]
Since $m_0\ne0$ and $\chi\ne0$, the square is nonzero in homological degree two, so this naive choice does not give a lift.

Any graded lift of $K$ has a linear differential. After choosing bases lifting those of $K$, the correction has the form $u\otimes\lambda\circ\fu+v\otimes\lambda\circ\fv$ for degree-one $k$-linear operators $\fu,\fv$ on $\cA$. Since $u^2,v^2,m_0$ are linearly independent, the square-zero condition requires $\fu^2=\fv^2=0$ and $\fu\fv+\fv\fu=-\mathrm r_\chi$. These are the relations of the Clifford algebra of the binary quadratic form $-\kappa st$ over $k[\kappa]$, with $\kappa$ acting by $\mathrm r_\chi$. We construct this action below, while the next lemma verifies directly that any such action will give a lift.

\begin{lemma}\label{lem:lift}
Let $\fu,\fv\colon \cA\to \cA$ be $k$-linear maps of degree one with
\[
\fu^2=0,\qquad\fv^2=0,\qquad\fu\fv+\fv\fu=-\mathrm r_\chi .
\]
Set $\cT_n=R\otimes_kV_n$ and
\[
d(1\otimes\lambda)=u\otimes\lambda\circ\fu+v\otimes\lambda\circ\fv+x\otimes\lambda\circ\mathrm r_\xi+y\otimes\lambda\circ\mathrm r_\eta\qquad(\lambda\in V_n).
\]
Then $\cT$ is a graded lift of $K$ along $\pi\colon R\to B=R/(u,v)$. Moreover $\cT^*\cong R\otimes_k\cA$, where $r\otimes a$ represents the $R$-linear map $1\otimes\lambda\mapsto\lambda(a)\,r$, and the dual differential is
\begin{equation}\label{eq:dualnode}
d^*(r\otimes a)=ru\otimes\fu a+rv\otimes\fv a+rx\otimes a\xi+ry\otimes a\eta .
\end{equation}
For $r\in R_j$ and $a\in \cA_m$, the corresponding element of $\Hom_{R}(\cT_m,R)$ has internal degree $j-m$.
\end{lemma}
\begin{proof}
The defining relations of $R$ give
\[
d^2(1\otimes\lambda)=u^2\otimes\lambda\fu^2+v^2\otimes\lambda\fv^2+m_0\otimes\lambda\Bigl(\fu\fv+\fv\fu+\sum_{i,j}q_{ij}\,\mathrm r_{\xi_i}\mathrm r_{\xi_j}\Bigr)=0,
\]
since $\mathrm r_{\xi_i}\mathrm r_{\xi_j}=\mathrm r_{\xi_j\xi_i}$ and $q$ is symmetric. Modulo $(u,v)$, the differential is $\bar d(1\otimes\lambda)=\bar x\otimes\lambda\circ\mathrm r_\xi+\bar y\otimes\lambda\circ\mathrm r_\eta$, which is the differential of $K$. Evaluation gives \eqref{eq:dualnode}, and the degree assertion follows since $1\otimes\lambda$ has degree $m$ while its value has degree $j$.
\end{proof}

Now let $k[\kappa]$ be a polynomial ring with $\deg\kappa=2$, and let $\cQ$ be the graded free $k[\kappa]$-module with basis $e,e_{\fu},e_{\fv},e_{\fu\fv}$ of degrees $0,1,1,2$. Define $k[\kappa]$-linear maps $\fu_{\cQ},\fv_{\cQ}\colon\cQ\to\cQ$ of degree one by
\[
\begin{aligned}
\fu_{\cQ}&\colon e\mapsto e_{\fu},\quad e_{\fu}\mapsto0,\quad e_{\fv}\mapsto e_{\fu\fv},\quad e_{\fu\fv}\mapsto0,\\
\fv_{\cQ}&\colon e\mapsto e_{\fv},\quad e_{\fv}\mapsto0,\quad e_{\fu}\mapsto-\kappa e-e_{\fu\fv},\quad e_{\fu\fv}\mapsto-\kappa e_{\fv}.
\end{aligned}
\]
Thus $\cQ$ is the left regular module of the Clifford algebra
\[
k[\kappa]\langle\fu_{\cQ},\fv_{\cQ}\rangle/
(\fu_{\cQ}^2,\fv_{\cQ}^2,\fu_{\cQ}\fv_{\cQ}+\fv_{\cQ}\fu_{\cQ}+\kappa),
\]
with $e_{\fu}=\fu_{\cQ}e$, $e_{\fv}=\fv_{\cQ}e$ and $e_{\fu\fv}=\fu_{\cQ}\fv_{\cQ}e$. The relations show that $1$, $\fu_{\cQ}$, $\fv_{\cQ}$ and $\fu_{\cQ}\fv_{\cQ}$ span this algebra, and their actions on $e$ are independent.

\begin{lemma}\label{lem:Q}
The operators on $\cQ$ satisfy the following properties:
\begin{enumerate}
\item $\fu_{\cQ}^2=0$, $\fv_{\cQ}^2=0$ and $\fu_{\cQ}\fv_{\cQ}+\fv_{\cQ}\fu_{\cQ}=-\kappa$.
\item The kernels and images are
\[
\begin{aligned}
\ker\fu_{\cQ}=\im\fu_{\cQ}&=k[\kappa]e_{\fu}\oplus k[\kappa]e_{\fu\fv},\\
\ker\fv_{\cQ}=\im\fv_{\cQ}&=k[\kappa]e_{\fv}\oplus k[\kappa](\kappa e+e_{\fu\fv}).
\end{aligned}
\]
\item If $a,a'\in\cQ$ satisfy $\fu_{\cQ}a=0$, $\fv_{\cQ}a'=0$ and $\fv_{\cQ}a+\fu_{\cQ}a'=0$, then $a=a'=0$.
\end{enumerate}
The same properties hold after tensoring with any graded $k$-vector space $E$, with operators $1_E\otimes\fu_{\cQ}$ and $1_E\otimes\fv_{\cQ}$ on $E\otimes_k\cQ$.
\end{lemma}
\begin{proof}
The first two assertions follow directly from the displayed action. For (3), applying $\fu_{\cQ}$ and $\fv_{\cQ}$ to $\fv_{\cQ}a+\fu_{\cQ}a'=0$ gives
\[
0=\fu_{\cQ}\fv_{\cQ}a=-\kappa a,\qquad
0=\fv_{\cQ}\fu_{\cQ}a'=-\kappa a',
\]
and multiplication by $\kappa$ is injective on $\cQ$. Tensoring over $k$ preserves these properties.
\end{proof}

\begin{lemma}\label{lem:clifford}
Let $E$ be a graded vector space with Hilbert series
\[
H_E(t)=\frac{1-t}{(1-2t)(1+t)}=1+2t^2+2t^3+6t^4+10t^5+22t^6+\cdots,
\]
so that $\dim E_0=1$ and $\dim E_n=(2^n+2(-1)^n)/3$ for $n\ge1$. There is an isomorphism of graded vector spaces $\Phi\colon E\otimes_k\cQ\to\cA$ with $\Phi(\kappa z)=\Phi(z)\chi$ for every $z\in E\otimes_k\cQ$.
\end{lemma}
\begin{proof}
The free algebra $\cA$ is a domain and $\chi\ne0$, so $\mathrm r_\chi$ is injective. Let $G\subseteq \cA$ be a graded complement of $\cA\chi$. The map $G\otimes_kk[\kappa]\to \cA$ given by $g\otimes\kappa^i\mapsto g\chi^i$, is surjective because $\cA=G+\cA\chi$ and $\cA$ is bounded below. It is also injective. Indeed, if $\sum_ig_i\chi^i=0$ with $g_i\in G$, then $g_0\in G\cap \cA\chi=0$, and cancelling $\chi$ on the right gives the claim by induction. Hence
\[
H_G(t)=H_{\cA}(t)(1-t^2)=\frac{1-t^2}{1-2t}=(1+t)^2H_E(t).
\]
The coefficients of $H_E$ are nonnegative, so there is a graded isomorphism $E\otimes(ke\oplus ke_{\fu}\oplus ke_{\fv}\oplus ke_{\fu\fv})\cong G$. Extending it $k[\kappa]$-linearly gives the desired $\Phi$.
\end{proof}

Fix $E$ and $\Phi$ as in Lemma~\ref{lem:clifford}, and define operators on $\cA$ by
\begin{equation}\label{eq:transportoperators}
\fu=\Phi(1_E\otimes\fu_{\cQ})\Phi^{-1} \qquad \text{and} \qquad
\fv=\Phi(1_E\otimes\fv_{\cQ})\Phi^{-1}.
\end{equation}
They each have degree one, and Lemma~\ref{lem:Q}(1), together with the compatibility of $\Phi$ with $\kappa$, gives $\fu^2=\fv^2=0$ and $\fu\fv+\fv\fu=-\mathrm r_\chi$.

\subsection{Dual cohomology and the coresolution}
Let $\cT$ be the lift of Lemma~\ref{lem:lift} determined by the operators in \eqref{eq:transportoperators}.

\begin{proposition}\label{prop:nodeD}
One has $H^0(\cT^*)=k\,m_0$. For $m\ge1$, every class in $H^m(\cT^*)$ is represented by a cocycle in $R(q)_1\otimes \cA_m$, and $\dim_kH^m(\cT^*)=\dim E_{m-1}$. The lift $\cT$ is dually small.
\end{proposition}
\begin{proof}
The differential of $\cT^*$ preserves internal degree $j-m$, so we compute the cohomology at $R_j\otimes\cA_m$ separately for each $j$. By \eqref{eq:dualnode} in view of the basis \eqref{eq:Rqbasis}, we observe that
\begin{align*}
d^*(1\otimes a)&=u\otimes\fu a+v\otimes\fv a+x\otimes a\xi+y\otimes a\eta,\\
d^*(u\otimes a+v\otimes a'+x\otimes g+y\otimes g')&=u^2\otimes\fu a+v^2\otimes\fv a'+m_0\otimes(\fv a+\fu a'+g\xi_1'+g'\xi_2'),\\
d^*(u^2\otimes a+v^2\otimes a'+m_0\otimes a'')&=u^3\otimes\fu a+v^3\otimes\fv a',\\
d^*(u^j\otimes a+v^j\otimes a')&=u^{j+1}\otimes\fu a+v^{j+1}\otimes\fv a'\qquad(j\ge3).
\end{align*}
Since $\xi_1',\xi_2'$ is a basis of $\cA_1$, every element of $\cA_m$ with $m\ge1$ can be written uniquely in the form $g\xi_1'+g'\xi_2'$ with $g,g'\in \cA_{m-1}$. By Lemma~\ref{lem:Q}(2) and \eqref{eq:transportoperators}, $\ker\fu=\fu(\cA)$ and $\ker\fv=\fv(\cA)$, so in particular $\fu$ and $\fv$ are injective on $\cA_0$.

Consider first when $j=0$. If $d^*(1\otimes a)=0$, then $a\xi=0$, so $a=0$.

Now consider when $j=1$. An element $z=u\otimes a+v\otimes a'+x\otimes g+y\otimes g'$ of $R_1\otimes \cA_m$ is a cocycle if and only if $\fu a=0$, $\fv a'=0$, and $\fv a+\fu a'+g\xi_1'+g'\xi_2'=0$. For $m=0$ this forces $a=a'=0$ and then $g=g'=0$. For $m\ge1$, the cocycles correspond bijectively to the pairs $(a,a')\in\ker\fu\times\ker\fv$ in $\cA_m$, because $(g,g')$ is then determined. The coboundaries $d^*(1\otimes f)$, $f\in \cA_{m-1}$, form a space of dimension $\dim \cA_{m-1}$, since $f\xi$ determines $f$. By Lemma~\ref{lem:Q}(2), $\ker\fu$ and $\ker\fv$ both have Hilbert series $H_E(t)(t+t^2)/(1-t^2)=tH_E(t)/(1-t)$. Hence
\[
\sum_{m\ge0}\dim H^m_{(j=1)}\,t^m=\frac{2tH_E(t)}{1-t}-\frac{t}{1-2t}=\frac{t(1-t)}{(1-2t)(1+t)}=t\,H_E(t).
\]

Next consider the case when $j=2$. A cocycle $u^2\otimes a+v^2\otimes a'+m_0\otimes a''$ has $a=\fu f$ and $a'=\fv f'$. Subtracting $d^*(u\otimes f+v\otimes f')$ leaves a cocycle $m_0\otimes a'''$. If $m\ge1$, write $a'''=g\xi_1'+g'\xi_2'$; then $m_0\otimes a'''=d^*(x\otimes g+y\otimes g')$. If $m=0$, then $a=a'=0$ and there are no coboundaries. So this part of the cohomology is $k(m_0\otimes1)$ for $m=0$ and zero for $m\ge1$.

Finally, consider when $j \ge 3$. A cocycle $u^j\otimes a+v^j\otimes a'$ has $\fu a=0=\fv a'$, so $a=\fu f$ and $a'=\fv f'$, and it equals $d^*(u^{j-1}\otimes f+v^{j-1}\otimes f')$.

This proves the first two assertions. For the dually small condition, note that $\rho_0(m_0)=\pi(m_0)=0$, so $\ker\rho_0=k\,m_0$. Let $m\ge1$, and let $z=u\otimes a+v\otimes a'+x\otimes g+y\otimes g'$ be a cocycle with $\rho_m[z]=0$. Reducing \eqref{eq:dualnode} modulo $(u,v)$ identifies $\Hom_B(K,B)$ with $B\otimes \cA$ having differential $\bar d^*(\beta\otimes a)=\beta\bar x\otimes a\xi+\beta\bar y\otimes a\eta$. The image of $z$ is $\bar x\otimes g+\bar y\otimes g'\in B_1\otimes \cA_m$, and the coboundaries in $B_1\otimes \cA_m$ are the elements $\bar d^*(1\otimes f)$, $f\in \cA_{m-1}$. Hence $g=f\xi$ and $g'=f\eta$ for some $f$. Then $z-d^*(1\otimes f)=u\otimes\alpha+v\otimes\alpha'$ is a cocycle without $x$- and $y$-components, so $\fu\alpha=0$, $\fv\alpha'=0$ and $\fv\alpha+\fu\alpha'=0$. Lemma~\ref{lem:Q}(3) and \eqref{eq:transportoperators} give $\alpha=\alpha'=0$. Thus $[z]=0$, and $\rho_m$ is injective.
\end{proof}
We are now ready for the proof of Theorem~\ref{thm:node}:
\begin{proof}[Proof of Theorem~\ref{thm:node}]
The ring data above satisfy the hypotheses of \Cref{subsec:framework}, and Proposition~\ref{prop:nodeD} verifies the dually small condition for the constructed lift. \Cref{thm:U} therefore implies that $k$ is an infinite graded syzygy, and so there is a minimal graded coresolution of the form:
\[
0\longrightarrow k\langle2\rangle\xrightarrow{\;1\mapsto m_0\;}R(q)\longrightarrow G^1\longrightarrow G^2\longrightarrow\cdots,
\]
where $G^n=\cF_n^*$ for the complex $\cF$ constructed in Section~\ref{sec:criterion} from this lift.

By Proposition~\ref{prop:nodeD}, $\fB_0=\varnothing$ and $|\fB_m|=\dim E_{m-1}$ for $m\ge1$. The representatives $\varphi_\nu\in R_1\otimes\cA_m$ have internal degree $1-m$. The pairing \eqref{eq:dual} therefore gives $\deg z_\nu=2-(1-m)=m+1$, so every letter, and hence every word, has weight equal to its length. Since $\cT_j=R(q)\langle j\rangle^{2^j}$, the construction gives
\[
\cF_n=R(q)\langle n\rangle^{\beta_n} \qquad \text{and} \qquad
G^n=R(q)\langle-n\rangle^{\beta_n}.
\]
The length series of the alphabet is $L(z)=\sum_{m\ge1}\dim E_{m-1}z^{m+1}=z^2H_E(z)$, and that of its words is $1/(1-L(z))$. Multiplying by the rank series $1/(1-2z)$ of $\cT$ gives
\[
\sum_n\beta_nz^n=\frac{1}{1-2z}\cdot\frac{1}{1-z^2H_E(z)}=\frac{(1+z)}{(1-2z)(1+z)-z^2(1-z)}=\frac{1+z}{1-z-3z^2+z^3},
\]
completing the proof.
\end{proof}

The twist is essential in these examples, as the following proposition shows.
\begin{proposition}\label{prop:nodetwist}
If $b_{12}=0$, then the minimal graded free resolution of $k$ over $B'$ has no graded lift along $\pi\colon R_b\to B'$, and $k$ is not a third syzygy over $R_b$. This applies in particular to the untwisted ring $R_0=S\ltimes N$.
\end{proposition}
\begin{proof}
Write $R=R_b$ and $\fa=\ann_R(N)$. When $b_{12}=0$, the space $\fa_1$ is spanned by $X-b_{11}U$ and $Y-b_{22}V$, so
\[
\fa_1R_1=k(z-b_{11}U^2)+k(z-b_{22}V^2) \qquad \text{and} \qquad
\dim_k\fa_1R_1\le2<3=\dim_kR_2.
\]

If $k$ were a third syzygy over $R$, Proposition~\ref{prop:necU} would give $\m^2=\fa\m$, since $S=R/N$ has positive depth and $\m_{B'}^2=0$. Since $R$ is standard graded and $\fa\subseteq\m$, the degree-two part of an equality $\m^2=\fa\m$ would read $R_2=\fa_1R_1$, which is impossible. 

If $\cT$ were a graded lift, then $I=d(\cT_1)$ would satisfy $I+\fa=\m$, since $H_0(\bar\cT)=R/(I+\fa)$ is $k$, and $I\cap\fa=I\fa$, since $H_1(\bar\cT)=0$, as in the proof of Proposition~\ref{prop:necU}. The last step of that proof would again give $\m^2=\fa\m$, which as we saw above cannot be.
\end{proof}

\section{Coresolutions as the dual of a resolution}\label{sec:graded}

Throughout this section we let $(R,\m,k)$ be a positively graded $k$-algebra of dimension $d$. For an $R$-module $N$, we let $H_N(t):=\sum_{i \in \Z} \dim_k(N_i)t^i$ denote the Hilbert series of $N$.
In this section we prove Theorem~\ref{thm:introC} as a consequence of the following more general result:

\begin{theorem}\label{thm:extlength}
Suppose that $\Ext_R^i(N,R)$ has finite length for every $i$ and vanishes for $i\gg0$. Then $N$ has finite length and
\[
\dim_k N=\sum_{i\ge0}(-1)^{i+d}\dim_k\Ext_R^i(N,R).
\]
In particular,
\[
\dim_kN\le\sum_{i\ge0}\dim_k\Ext_R^i(N,R).
\]
\end{theorem}

\begin{proof}
Set $E(t)=\sum_{i \ge 0} (-1)^iH_{\Ext^i_R(N,R)}(t)$, and note that by assumption $E(t)$ is a Laurent polynomial in $t$. It follows from \cite[Theorem 1]{ABS} that there is an equality of rational functions
\[E(t)=\dfrac{H_{N}(t^{-1})H_R(t)}{H_R(t^{-1})}.\]
By \cite[Proposition 4.4.1]{BH}, we may write $H_R(t)=\dfrac{\epsilon_R(t)}{\prod_{i=1}^d (1-t^{a_i})}$ for some positive integers $a_1,\dots,a_d$ and with $\epsilon_R(t) \in \Z[t,t^{-1}]$ satisfying $\epsilon_R(1)>0$. But then from the above, we have 
\[H_{N}(t^{-1})\epsilon_R(t)\prod^d_{i=1} (1-t^{-a_i})=E(t)\prod^d_{i=1} (1-t^{a_i})\epsilon_R(t^{-1})=(-1)^dt^{\sum^d_{i=1} a_i}\epsilon_R(t^{-1})E(t)\prod^d_{i=1}(1-t^{-a_i}),\] 
which gives
\[\epsilon_R(t)H_{N}(t^{-1})=(-1)^dt^{\sum^d_{i=1} a_i}\epsilon_R(t^{-1})E(t).\]
In particular, $H_{N}(t)$ has no pole at $t=1$, and appealing again to \cite[Proposition 4.4.1]{BH} implies that $\dim(N)=0$, so $H_{N}(t)$ is a Laurent polynomial. Since $\epsilon_R(1) \ne 0$, evaluating at $t=1$ gives 
\[\epsilon_R(1)\dim_k N=\epsilon_R(1)\sum_{i\ge0}(-1)^{i+d}\dim_k\Ext_R^i(N,R).\]
Since $\epsilon_R(1) \ne 0$, the claims follow.
\end{proof}

If $\Hom_R(N,R)\cong k$ and $\Ext^i_R(N,R)=0$ for $i>0$, then a minimal free resolution of $N$ dualizes to a minimal coresolution of $k$. The following result shows that, allowing for a graded shift, this approach cannot produce positive-dimensional examples in the graded setting.

\begin{theorem}\label{thm:graded}
Suppose $(R,\m,k)$ is a positively graded $k$-algebra. If there is an $R$-module $N$ and an integer $\ell$ such that $\Hom_R(N,R) \cong k\langle \ell \rangle$ and $\Ext^{i>0}_R(N,R)=0$, then $R$ is an Artinian Gorenstein ring. 
\end{theorem}

\begin{proof}
By assumption we have $\sum_{i \ge 0} \dim_k \Ext^i_R(N,R)=1$, and it follows from \Cref{thm:extlength} that $\dim_k N=1$. This forces $N \cong k$ up to graded shift, but then $\Ext^{i>0}_R(k,R)=0$ so $R$ is Artinian Gorenstein. 
\end{proof}

\begin{remark}\label{remark:dualnotacyclic}
\Cref{thm:graded} shows that in the graded setting, the coresolutions of Section~\ref{sec:criterion} cannot come from dualizing a resolution unless $R$ is Artinian Gorenstein to begin with.
\end{remark}

\section*{Acknowledgements}
We are grateful to Srikanth Iyengar for introducing us to this problem, and for many helpful discussions. The first author was supported by the National Natural Science Foundation of China (No. 12401046).

\section*{AI Use Disclosure}
Various AI models were used in the preparation of this work for ideation, drafting, and proofreading. The authors take responsibility for all content.

\vspace{.3in}


\begin{thebibliography}{99}

\bibitem{Mohsen}
M.~Asgharzadeh, \emph{Reflexivity revisited},
preprint, \href{https://arxiv.org/abs/1812.00830}{arXiv:1812.00830}
(2018), revised 2025.


\bibitem{AMS} R.~Ahangari Maleki and L.~M. \c{S}ega, \emph{The absolutely Koszul and Backelin--Roos properties for spaces of quadrics of small codimension}, J. Algebra \textbf{551} (2020), 232--284.

\bibitem{Avramov} L.~L. Avramov, \emph{Infinite free resolutions}, in: Six lectures on commutative algebra (Bellaterra, 1996), Progr. Math. \textbf{166}, Birkh\"auser, Basel, 1998, 1--118.

\bibitem{ABS} L.~L. Avramov, R.-O. Buchweitz and J.~D. Sally,
\emph{Laurent coefficients and Ext of finite graded modules},
Math. Ann. \textbf{307} (1997), 401--415.

\bibitem{BH} W.~Bruns and J.~Herzog,
\emph{Cohen--Macaulay rings}, revised ed.,
Cambridge Stud. Adv. Math. \textbf{39},
Cambridge University Press, Cambridge, 1998.

\bibitem{Gulliksen} T.~H. Gulliksen, \emph{A proof of the existence of minimal $R$-algebra resolutions}, Acta Math. \textbf{120} (1968), 53--58.

\bibitem{Lofwall} C.~L\"ofwall, \emph{On the subalgebra generated by the one-dimensional elements in the Yoneda Ext-algebra}, in: Algebra, algebraic topology and their interactions (Stockholm, 1983), Lecture Notes in Math. \textbf{1183}, Springer, Berlin, 1986, 291--338.

\bibitem{MRS} C.~Miller, H.~Rahmati and J.~Striuli, \emph{The residue field as a high syzygy}, Comm. Algebra \textbf{46} (2018), no.~6, 2620--2630.

\bibitem{PP} A.~Polishchuk and L.~Positselski, \emph{Quadratic algebras}, Univ. Lecture Ser. \textbf{37}, Amer. Math. Soc., Providence, RI, 2005.

\bibitem{Priddy} S.~B. Priddy, \emph{Koszul resolutions}, Trans. Amer. Math. Soc. \textbf{152} (1970), 39--60.

\bibitem{Schoeller} C.~Schoeller, \emph{Homologie des anneaux locaux noeth\'eriens}, C. R. Acad. Sci. Paris S\'er. A \textbf{265} (1967), 768--771.

\bibitem{Tate} J.~Tate, \emph{Homology of Noetherian rings and local rings}, Illinois J. Math. \textbf{1} (1957), 14--27.

\end{thebibliography}
\end{document}